\documentclass[12pt,reqno]{amsart}

\usepackage{latexsym, amsmath, amscd, amssymb, amsthm}
\usepackage{bm}
\usepackage{mathrsfs}   
\usepackage{xspace}
\usepackage{enumerate}
\usepackage{verbatim}
\usepackage[english]{babel}
\usepackage{epsfig}
\newcommand{\R}{{\mathbb R}}

\begin{document}
 
\newtheorem{lemma}{Lemma}[section]
\newtheorem{theorem}[lemma]{Theorem}
\newtheorem{corollary}[lemma]{Corollary}
\newtheorem{proposition}[lemma]{Proposition}
\theoremstyle{definition}
\newtheorem{definition}[lemma]{Definition}
\newtheorem{conjecture}[lemma]{Conjecture}
\newtheorem{question}[lemma]{Question}
\newtheorem{problem}[lemma]{Problem}
\newtheorem{claim}[lemma]{Claim}
\newtheorem{example}[lemma]{Example}
\newtheorem{remark}[lemma]{Remark}
\newtheorem{assumption}[lemma]{Assumption}
\newtheorem*{acknowledgements}{Acknowledgments}
\theoremstyle{remark}

\title[Spectral bounds ]{Spectral bounds for non-uniform hypergraphs using weighted clique expansion }  
\author[A. Guha]{Ashwin Guha}
\address{Department of Computer Science and Automation \\ Indian Institute of Science \\ Bangalore 560012, India.}
\email{guha.ashwin@gmail.com }

\author[A. Dukkipati]{Ambedkar Dukkipati}
\address{Department of Computer Science and Automation \\ Indian Institute of Science \\ Bangalore 560012, India.}
\email{ambedkar@iisc.ac.in}

\begin{abstract}

Hypergraphs are an invaluable tool to understand many hidden patterns in large data sets. Among many ways to represent hypergraph, one useful representation is that of weighted clique expansion. In this paper, we consider this representation for non-uniform hypergraphs. We generalize the spectral results for uniform hypergraphs to non-uniform hypergraphs and show that they extend in a natural way. We provide a bound on the largest eigenvalue with respect to the average degree of neighbours of a vertex in a graph. We also prove an inequality on the boundary of a vertex set in terms of the largest and second smallest eigenvalue and use it to obtain bounds on some connectivity parameters of the hypergraph.

\end{abstract}

\maketitle
 
\section{Introduction}
\label{sec_intro}

Hypergraphs are a generalization of combinatorial graphs, where an edge may span three or more vertices. Hypergraphs can be used to model complex interactions among various entities. They are applied in a wide range of areas from protein interaction \cite{ramadan2004hypergraph}, social network analysis \cite{qian2009hypergraph} to image processing \cite{ducournau2012reductive}. Hypergraphs are now widely used in machine learning. One of the key problems is that of hypergraph partitioning or clustering. Several algorithms have been proposed for clustering \cite{bulo2009game,leordeanu2012efficient,papa2007hypergraph}. Among these, spectral methods are proving to be popular \cite{zhou2007learning, agarwal2006higher,ghoshdastidar2017consistency}. Typically, in these algorithms, we define a Laplacian matrix. We then use the use the top $k$ eigenvalues and corresponding eigenvectors to construct $k$-clusters.  

It is possible to represent a hypergraph in many ways. One common method is to use a tensor to describe a hypergraph. Lim \cite{lim2005singular} and Qi \cite{qi2005eigenvalues} independently defined eigenvalues for a tensor. Thus, one may define an adjacency tensor and a Laplacian tensor for a $k$-graph and use eigenvalues of these tensors to obtains spectral bounds. It is also possible to use tensors for non-uniform hypergraphs using a modified definition of adjacency tensor  \cite{cooper2012spectra,li2017analytic, qi2014h+, yang2010further, banerjee2017spectra}. 

There are also many matrix representations of hypergraphs. The simplest way is by using clique expansion, where one converts the hypergraph into a 2-graph by considering each hyperedge of size $k$ as a $k$-clique in a 2-graph. This representation, however, does not capture the full information of the hypergraph. Two different hypergraphs may give rise to the same 2-graph. 

\begin{example}
Consider two non-isomorphic 3-graphs $G_1= \{123, 124,234 \}$ and $G_2 =\{123,124,134,234 \}$. They both give rise to the same 
2-graph.
 \begin{center}
  \includegraphics[height=3cm]{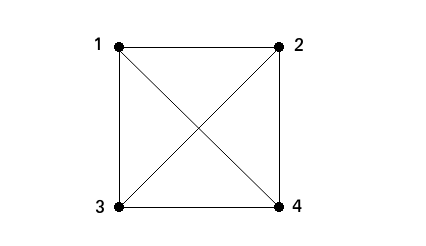}
 \end{center}
\end{example}

A modified version of clique expansion was introduced by Rodriguez in \cite{rodriguez2003laplacian}. A weighted 2-graph is obtained from the hypergraph in this expansion. 

Formally, a hypergraph $H(V,E)$ is a collection of vertices $V$ and edges $E$, which are subsets of $V$, \emph{i.e.} $E \subset 2^V$. If all edges in $E$ are of same size $k$, then $H$ is called a \emph{$k$-uniform hypergraph} or a \emph{$k$-graph}. The \emph{degree} of a vertex is defined as the number of edges containing the vertex. Let $n$ be the number of vertices in $H$. We define an $n \times n$ \emph{adjacency matrix} $A$ such that for all pairs of vertices $i$ and $j$, $a_{ij}$ is the number of edges containing $i$ and $j$.  

\[A_{ij}=
\begin{cases}
 0 & \text{ if } i=j, \\
 |\{e_{\ell}\}| & \text{ if }  i,j \in e_{\ell},\\
 0 &\text{ otherwise.}
\end{cases} \]

For a 2-graph, this definition coincides with the familiar definition of adjacency matrix.  

We define \emph{Laplacian degree} of a vertex as $\delta_i = \displaystyle\sum\limits_{j=1}^{n} a_{ij}$ \cite{rodriguez2009laplacian}. For a $k$-graph, we have $\delta_i=(k-1)d_i$. For a general hypergraph, $\delta_i = \displaystyle\sum\limits_{e \supset i} |e|-d_i $, which gives us the inequality, $(k_{\min}-1)d_i \leq \delta_i \leq (k_{\max}-1)d_i. $
We may observe that the vertex with maximum degree need not be the vertex with maximum $\delta$. The Laplacian of a hypergraph is defined as 
\[L_{ij}=
\begin{cases}
 \delta_{i} & \text{ if } i=j, \\
 -a_{ij}  &\text{ otherwise.}
\end{cases} \]

Let $\lambda_1 \leq \lambda_2 \leq \ldots \leq \lambda_n$ be the eigenvalues of $L$. Since $L$ is symmetric, the eigenvalues are real and non-negative. In particular, $\lambda_1=0 $ and $\lambda_2 > 0 $ if and only if the hypergraph is connected.

In \cite{rodriguez2009laplacian}, bounds are established for the second smallest eigenvalue of the Laplacian of a $k$-graph in terms of the size of subsets. These results in turn help us bound many connectivity parameters of the $k$-graph. 

In this paper, we wish to extend these results for non-uniform hypergraphs. The definitions mentioned above were initially given for $k$-graphs, but they hold for general hypergraphs without any modification. Hence, we can get similar bounds for non-uniform graphs as well. We show that the results extend in an intuitive way. 

This paper is organized as follows. In Section \ref{sec_prelimresults}, we give the spectrum of a complete $k$-graph, a complete $k$-partite graph and a star graph in this representation. We then proceed to give simple bounds on $\lambda_n $. We also extend a bound on $\lambda_n$ in terms of average degree of neighbours from 2-graphs to general hypergraphs. Section \ref{sec_connresults} contains the bounds of various connectivity parameters in terms of $\lambda_2 $ and $\lambda_n $. Section \ref{sec_summary} contains the summary and concluding remarks. 

\section{Preliminary Results}
\label{sec_prelimresults} 

We now consider the spectrum of some simple $k$-graphs. These graphs are useful in the sense that they provide an upper bound for the spectrum. In many spectral problems, these are the cases which represent the extremal cases, hence worthwhile to study. 

\begin{proposition}
 The Laplacian spectrum of complete $k$-graph with $n$ vertices is $0$ with multiplicity one and $n \binom{n-2}{k-2}$ with multiplicity  $(n-1)$.
\end{proposition}

\begin{proof}
The Laplacian matrix is as follows. 
\[L_{ij}=
\begin{cases}
 (n-1) \binom{n-2}{k-2} & \text{ if } i=j, \\
  -\binom{n-2}{k-2} &\text{ otherwise.}
\end{cases} \]
The vector $(1,\ldots,1)^T$ corresponds to eigenvalue $0$ and $(1,-1,0,\ldots,0)^T$ through $(1,0,\ldots,0,-1)^T$ correspond to eigenvalue $n \binom{n-2}{k-2}$.
\end{proof}

It is possible to get a similar result for a $k$-partite graph, \emph{i.e.} a graph where the vertex set is divided into $k$-partitions and each edge consists of exactly one vertex from each partition. 

\begin{proposition}
For a complete $k$-partite graph with partitions of size $n_1, \ldots, n_k $, the Laplacian spectrum consists of  
$0$ with multiplicity one, $(k-1)\left( \displaystyle\prod\limits_{j=1}^{k} n_j\right) /n_i $  with multiplicity $(n_i -1) $ for all $i=1, \ldots, k$. The remaining $(k-1) $ eigenvalues are roots of the polynomial
$$ X^{k-1} - A_{k-2} X^{k-2} + \ldots + (-1)^{k-1}A_0, $$ 
where the coefficient $A_i$, the sum of products of roots taken $(k-1-i)$ at a time, is given by
$$A_i = (i+1)k^{k-2-i} \left(\prod_{j=1}^{k}n_j\right)^{k-2-i} \left( \sum_{1 \leq j_1 \ldots j_{i+1} \leq k} n_{j_1} \ldots n_{j_{i+1}} \right) .$$
\end{proposition}

For example, consider a 5-partition where the sizes of the partitions are $a,b,c,d,e$. Then the eigenvalues are 
$0$ with multiplicity one, $4abcd$ with multiplicity $(e-1)$, $4abce$ with multiplicity $(d-1)$, $4abde$ with multiplicity $(c-1)$,
$4acde$ with multiplicity $(b-1)$ and $4bcde$ with multiplicity $(a-1)$. The remaining four eigenvalues are the roots of the polynomial 

\begin{align*} 
X^4 &- 4\left(\sum abcd \right)X^3 + 3 \cdot 5\left(\prod a\right)\left(\sum abc\right)X^2 \\
&- 2 \cdot 5^2\left(\prod a\right)^2 \left(\sum ab\right)X + 1 \cdot 5^3 \left(\prod a\right)^3\left(\sum a\right) .
\end{align*}

For a star $k$-graph with $r$ spokes, the spectrum is obtained easily. The number of vertices in such a graph is given by $n=(k-1) r+1 $. 
\begin{proposition}
 The Laplacian spectrum of a star $k$-graph with $r$ spokes consists of $0$ and $n$ with multiplicity one each, $1$ with multiplicity $(r-1)$, and $k$ with multiplicity $(k-2)r$.
\end{proposition}

The eigenvectors are given below. 

\begin{align*}
0 &:  (1, \ldots,1)^T, \\ 
n &: (n-1| -1,\ldots,-1)^T,\\
1 &:(0| 1, \ldots,1 | -1, \ldots,-1 |0, \ldots,0)^T, \ldots, (0| 1, \ldots,1 | 0, \ldots, 0 |-1, \ldots,-1)^T,\\
k &:(0| 1,-1,0, \ldots,0| 0,\ldots,0)^T, \ldots, (0| 1,0, \ldots,0,-1| 0,\ldots,0)^T,\ldots, \\
  &:(0| 0,\ldots,0 |1,-1,0, \ldots,0)^T, \ldots,(0| 0,\ldots,0 | 1,0, \ldots,0,-1)^T. 
\end{align*}

It is possible to get some simple bounds on the largest eigenvalue of the Laplacian. These are generalizations of the bounds for 2-graphs. The results mentioned here are valid for both uniform as well as non-uniform hypergraphs. The proofs are exactly the same as for weighted graphs. 


\begin{proposition}
$$\lambda_n \leq 2 \max_i \delta_i.$$ 
\end{proposition}
\begin{proof}
Let $x$ be the eigenvector of $\lambda_n$. Let $x_i$ be the largest component, \emph{i.e.}. $|x_i| \geq |x_j|$ for all $j=1, \ldots, n$. Assume $0< x_i \leq 1$.
\begin{align*}
\lambda x_i = (Lx)_i &= \delta_ix_i - \sum_{i \sim j} a_{ij}x_j \\
&\leq \delta_ix_i + \sum_{i \sim j} |a_{ij} x_j| \\
&\leq \delta_ix_i + \sum_{i \sim j} |a_{ij} x_i| \\
&\leq 2 \delta_i x_i \leq 2 \max_i \delta_i x_i.
\end{align*}
\end{proof}

\begin{proposition}
\label{degi_plus_degj}
$$\lambda_n \leq  \max_{i \sim j} (\delta_i + \delta_j).$$
\end{proposition}
\begin{proof}
Let $\Delta= \mbox{ diag}(\delta_1, \ldots, \delta_n) $. Consider $\Delta^{-1}L \Delta$.
\[
(\Delta^{-1}L \Delta)_{ij}= 
  \begin{cases} 
   \delta_i & \text{if } i=j, \\
   -a_{ij}\frac{\delta_j}{\delta_i}     & \text{otherwise. } 
  \end{cases}
\]
Applying Gerschgorin theorem, there exists an $i$ such that 
$$|\lambda_n - \delta_i| \leq \sum_{i \sim j } \left|a_{ij}\frac{\delta_j}{\delta_i}\right|= \frac{\sum_j a_{ij} \delta_j}{\sum_{j}a_{ij} } \leq \max_{i \sim j} \delta_j. $$
Hence
$$ \lambda_n \leq  \max_{i \sim j} (\delta_i + \delta_j). $$
\end{proof}

Let $m_i$ denote the average degree of neighbours of vertex $i$, \emph{i.e.} $ m_i=\frac{\sum_{i\sim j} d_j}{d_i}$. There are many upper bounds on the largest eigenvalue in terms of $m_i$. (For example, see \cite{aouchiche2010survey} and the references therein). We have the following  bound for 2-graphs by Zhu \cite{zhu2010upper} that relates the largest eigenvalue with the degree and average degree of neighbours of each vertex. 
\begin{theorem}
\label{zhu_theorem}
$$\lambda_n \leq \max_{i \sim j } \left\lbrace \frac{d_i(d_i + m_i)+d_j(d_j + m_j) - 2 \sum_{\ell \in N(i) \cap N(j)}d_{\ell} }{d_i + d_j}\right\rbrace.$$
\end{theorem}
\noindent Here $N(i)$ is the set of vertices adjacent to $i$. This result holds for $k$-uniform hypergraphs in the same form. For non-uniform hypergraphs there is 
an additional factor of $\left(\frac{k_{\max}-1}{k_{\min}-1} \right)$. In order to prove the result for non-uniform case 
we make use of the lemma below (Theorem 2.3 in \cite{zhu2010upper}).

\begin{lemma}
\label{zhu_lemma}
Let $G = (V,E)$ be a simple graph. Let $f : V \times V \longrightarrow \R^+  \cup \{0\}$ be a nonnegative function
which is positive on edges. Then $\lambda_n$ is less than or equal to
$$ \max_{i \sim j} \left\lbrace |N(i) \cap N(j)| + \frac{\sum_{\ell \in N(i)\backslash N(j)}f(i,\ell)+ \sum_{\ell \in N(j)\backslash N(i)} f(j,\ell)}{f(i,j)} \right\rbrace.$$
\end{lemma}

The generalization of \ref{zhu_theorem} for non-uniform hypergraphs is presented below. 

\begin{theorem}
 $$\lambda_n \leq \max_{i \sim j } \left\lbrace \left(\frac{k_{\max}-1}{k_{\min}-1}\right)\left( \frac{d_i(d_i + m_i)+d_j(d_j + m_j) - 2 \sum_{\ell \in N(i) \cap N(j)}d_\ell }{d_i + d_j}\right)\right\rbrace.$$
\end{theorem}

\begin{proof}
  Substituting $f(i,j)= \delta_i + \delta_j$ in Lemma \ref{zhu_lemma}, we get the term
 \begin{align*}
  &\sum_{\ell \in N(i)\backslash N(j)} (\delta_i + \delta_{\ell} )+ \sum_{\ell \in N(j)\backslash N(i)} (\delta_j + \delta_{\ell}) \\
  &= \sum_{N(i)} (\delta_i + \delta_{\ell}) + \sum_{N(j)} (\delta_j + \delta_{\ell}) - \sum_{N(i)\cap N(j)} (\delta_i + \delta_j + 2 \delta_{\ell}).
 \end{align*}
Then 
\begin{align*}
&|N(i)\cap N(j)|+ \frac{\sum_{N(i)} (\delta_i + \delta_{\ell}) + \sum_{N(j)} (\delta_j + \delta_{\ell})}{\delta_i + \delta_j} \\
&- |N(i)\cap N(j)|\frac{(\delta_i + \delta_j)}{(\delta_i + \delta_j)} - \frac{2 \sum_{N(i)\cap N(j)} \delta_{\ell}}{\delta_i + \delta_j} \\
&= \frac{ \sum_{N(i)} (\delta_i + \delta_{\ell}) + \sum_{N(j)} (\delta_j + \delta_{\ell})- 2 \sum_{N(i)\cap N(j)} \delta_{\ell}}{\delta_i + \delta_j}.
\end{align*}

Since $\delta_i \leq (k_{\max}-1)\:d_i $, for all $i =1, \ldots, n$, for the numerator we have,
\begin{align*}
&\sum_{N(i)} (\delta_i + \delta_{\ell}) + \sum_{N(j)} (\delta_j + \delta_{\ell})- 2 \sum_{N(i)\cap N(j)} \delta_{\ell}  \\
&\leq (k_{\max}-1)(d_i(d_i + m_i)+d_j(d_j + m_j) - 2 \sum_{{\ell} \in N(i) \cap N(j)}d_{\ell}) 
\end{align*}

The denominator of the expression becomes 
$$\delta_i + \delta_j \geq (k_{\min}-1)(d_i + d_j), $$ since $ \delta_i \geq (k_{\min}-1)\:d_i$. 
Together with the above expression we get the required result.
\end{proof}

One can observe that each of the results mentioned is an improvement over the preceding bound. Not all such results, however, translate from 2-graphs to hypergraphs. For example consider the statement from Proposition \ref{degi_plus_degj}.  

For 2-graphs we have, $$ \lambda_n \leq \max_{i \sim j} (d_i + d_j).$$  However, this statement does not extend 
to hypergraphs in the form $$ \lambda_n \leq \max_{i_1\ldots i_r \in E} (d_{i_1}+ \ldots + d_{i_r}).$$ 
The counterexample is as follows.

\begin{example}
 Consider the graph $H= \{123,124,235,345 \} $. $\lambda_n =8.23 > 8 = \{d_2 + d_3 + d_5 \} $.
\end{example}

\section{Connectivity Results}
\label{sec_connresults}

We now present some connectivity results for general hypergraphs. Let $\partial S$ denote the edge boundary of $S$
which is defined as $ \partial S= E(S, V \backslash S)$. Various connectivity parameters can be defined in terms of $\partial S$. Some results have been provided for uniform hypergraphs in \cite{rodriguez2009laplacian}. We provide the same for non-uniform hypergraphs.  

First, we prove the following lemma. 

\begin{lemma}
\label{boundary_bound}
 For any $S \subset V $, 
 $$ \frac{4 \lambda_2 |S| (n-|S|)}{nk^2_{\max}} \leq |\partial S| \leq \frac{\lambda_n |S| (n-|S|)}{n(k_{\min}-1)}.$$
\end{lemma}

\begin{proof} 
The bound holds for extreme cases $S =\emptyset $ and $S=V$. Let $S$ be a proper subset of $V$. Let $\chi_S$ be the indicator vector of $S$, 
$$  \chi_S =\begin{cases}  1 & \text{ if } i\in S, \\ 0  &\text{ otherwise.}
\end{cases} $$
Then $ \displaystyle\sum\limits_{i \in V} \displaystyle\sum\limits_{j \in V} (w_i-w_j)^2= 2|S|(n-|S|)$. We have the following results by 
Fielder \cite{fiedler1973algebraic}.

$$\lambda_2 = 2n \min \left\lbrace\frac{\sum_{i \sim j}a_{ij} (w_i - w_j)^2 }{\sum_{i \in V} \sum_{j \in V} (w_i - w_j)^2} : w \neq c \cdot 1_n \mbox{ for } c \in \R\right\rbrace .$$
$$\lambda_n = 2n \max \left\lbrace\frac{\sum_{i \sim j}a_{ij} (w_i - w_j)^2 }{\sum_{i \in V} \sum_{j \in V} (w_i - w_j)^2} : w \neq c \cdot 1_n \mbox{ for } c \in \R\right\rbrace . $$
Thus we have,
$$ \lambda_2 \leq \frac{n \sum_{i \sim j} a_{ij}(w_i-w_j)^2 }{|S|(n-|S|)} \leq \lambda_n.$$

In the sum $(w_i - w_j)^2$, only the boundary edges contribute to the sum. Let $e \in \partial S$ such that $|e \cap S| =k$. The edge $e$ contributes 
$k(|e|-k)$ to the sum. The maximum value of this function over all the edges in $E$ is $k_{\max}^2 /4$, and the minimum value is $(k_{\min}-1)$. Substituting we get the required inequality.
\end{proof}

The following results directly follow form Lemma \ref{boundary_bound} above.

\begin{theorem}
\label{edge_density} 
Let the edge-density of a set $S \subset V $ be defined as $\rho(S) = \frac{|\partial S|}{|S|(n-|S|)} $. Then,
$$ \frac{4 \lambda_2}{nk^2_{\max}} \leq \rho(S) \leq \frac{\lambda_n }{n(k_{\min}-1)} .$$
\end{theorem}

\begin{theorem}
Let max-cut be defined as $mc(H)= \max \{|\partial S|: S \subset V \} $. Then,
$$mc(H) \leq \frac{n \lambda_n}{4(k_{\max}-1)}. $$
\end{theorem}

\begin{theorem}
\label{mat_isoperi}
Let the isoperimetric number be defined as $\phi(H)= \min_{S \subset V} \{ \frac{|\partial S|}{|S|}: |S| \leq n/2 \} $. Then,
$$\phi(H) \geq \frac{2\lambda_2}{k^2_{\max}} .$$
\end{theorem}

Note that Theorem \ref{mat_isoperi} gives us one half of the famous Cheeger inequality.

It appears that the introduction of terms $k_{\min} $  and $k_{\max}$ necessarily slackens the bounds. The equality is attained only in $k$-uniform hypergraphs. 


\begin{example}
Consider the following 3-graph $H= \{123,234,456,156 \} $ .
\begin{center}
\includegraphics[height=2.5cm]{matbound_unif.png} 
\end{center}
The spectrum of the Laplacian is $0^1, 2^1, 4^1, 6^3$. The set $A = \{1,4\}$ attains the upper bound of Lemma \ref{boundary_bound}. 
$$|\partial A|=4=  \left\lfloor  \frac{6\cdot 2\cdot 4}{2 \cdot 6 }  \right\rfloor. $$
The set $B= \{1,2,3 \}$ attains the lower bound.
$$|\partial B| =2= \left\lceil \frac{4 \cdot 2 \cdot 3 \cdot 3}{8 \cdot 6}  \right\rceil.$$

Now consider the following non-uniform graph $H' = \{123,456,34,1256 \} $.
\begin{center}
\includegraphics[height=2.5cm]{matbound_nonunif.png} 
\end{center}
The spectrum of the Laplacian is $0^1, 3^2, 6^1,7^2 $. The set $A = \{1,2,4 \}$ attains the maximum.
$$|\partial A|=4 < \left\lfloor \frac{7 \cdot 3 \cdot 3}{1 \cdot 6}  \right\rfloor. $$
Similarly, the singleton set $B = \{ 3 \} $ attains the minimum.
$$|\partial B|=2 > \left\lceil \frac{4 \cdot 3 \cdot 1 \cdot 5}{4\cdot 4\cdot 6} \right\rceil. $$
Note that if we had used $k_{\min}^2$ in the denominator instead of $k_{\max}^2$, the lemma would no longer be true.

\end{example}

\section{Concluding Remarks}
\label{sec_summary}

 In this paper we have generalized the results for $k$-graphs to non-uniform hypergraphs. We have established bounds for the eigenvalues of the Laplacian and some connectivity results. We have also shown that the bounds are no longer tight for non-uniform hypergraphs. 

It must be noted that the crucial property of uniqueness in representation does not translate from uniform to general hypergraphs. This fact, however, does not invalidate our results. Two different hypergraphs may give rise to the same Laplacian matrix, and hence the same spectrum. But the values of $|\partial S|$, $k_{\min}$ and $k_{\max}$ will be different for those hypergraphs and consequently the results proved here will still hold. Hence, the results are valid, though their application may differ depending on the kind of hypergraph in question.

General hypergraphs offer greater flexibility for many applications, but there is a dearth of literature devoted to non-uniform hypergraphs. We hope to further explore hypergraphs using different representations in the future.


\end{document}